\documentclass[12pt]{amsart}
\usepackage{amsmath}
\usepackage{enumerate}
\usepackage{amssymb}
\usepackage{amsbsy}
\usepackage{amsfonts}
\usepackage[arrow]{xy}

\allowdisplaybreaks
\theoremstyle{plain}

\newtheorem{thm}{Theorem}[section]

\newtheorem{lem}[thm]{Lemma}

\theoremstyle{definition}

\newtheorem{rmk}[thm]{Remark}
\numberwithin{equation}{section}
\newcommand{\sm}{\left(\begin{smallmatrix}}
\newcommand{\esm}{\end{smallmatrix}\right)}

\newfont{\FieldFont}{msbm10 scaled\magstep1}

\def\re{\textrm{Re}}
\def\im{\textrm{Im}}
\def\tr{\textrm{Tr}}

\usepackage{color}

\definecolor{blue}{rgb}{0,0,1}
\definecolor{red}{rgb}{1,0,0}
\definecolor{green}{rgb}{0,.6,.2}
\definecolor{purple}{rgb}{1,0,1}

\long\def\red#1\endred{{\color{red}#1}}
\long\def\blue#1\endblue{{\color{blue}#1}}
\long\def\purple#1\endpurple{{\color{purple}#1}}
\long\def\green#1\endgreen{{\color{green}#1}}

\begin{document}

\title{Kernels for Products of Hilbert L-functions}

 \author{YoungJu Choie}

 \address{Department of Mathematics  \\
 Pohang University of Science and Technology (POSTECH)\\
 Pohang, 790--784, Korea}
  \email{yjc@postech.ac.kr; http://yjchoie.postech.ac.kr }

\author{Yichao Zhang}

 \address{School of Mathematics and Institute of Advanced Studies of Mathematics\\
 Harbin Institute of Technology\\ Harbin, 150001, P.R.China}
  \email{yichao.zhang@hit.edu.cn}


\thanks{Keynote: Hilbert modular form, special $L$-values, cusp form, double Eisenstein series, Petersson inner product, Rankin-Cohen bracket, kernel function}

\begin{abstract}
 We study kernel functions of $L$-functions and products of $L$-functions
 of Hilbert cusp forms over real quadratic fields.
 This extends the results on elliptic modular forms in \cite{CD1, CD2}. 
\end{abstract}
 \maketitle

\section{Introduction  }

One of the central problems in number theory is to explore the nature of special values of various 
Dirichlet series such as Riemann zeta function, modular $L$-functions, automorphic $L$-functions, etc. The known main idea to study arithmetic properties of the special values of modular $L$-functions is to  compare such values with certain inner product of modular forms.

Such an idea was first introduced by  Rankin \cite{R}, expressing the product of two critical $L$-values of an elliptic Hecke eigenform in terms of the Petersson scalar product of an elliptic Hecke eigenform  with a product of Eisenstein series.  
Much later Zagier (\cite{Z1977}, p 149 )  extended Rankin's result to express the product of any two critical $L$-values of an elliptic Hecke eigenform  in terms of the Petersson scalar product of the Hecke eigenform with the   Rankin-Cohen brackets  of two Eisenstein series. 
Shimura \cite{Sh} and Manin \cite{M}  developed theories to study arithmetic properties of   modular $L$-values on the critical strip.  Kohnen-Zagier   and Choie-Park-Zagier    \cite{KZ, CPZ} further studied the  space of modular forms whose $L$-values on the critical strip are rational and showed that such a space can be spanned by Cohen kernel introduced by Cohen \cite{C}. 
Recently double Eisenstein series has been introduced by Diamantis and O'Sullivan\cite{CD1, CD2} as a kernel yielding products of two $L$-values of elliptic Hecke eigenforms. It turns out that Rankin-Cohen brackets \cite {Z} of two Eisenstein series can be realized as a double Eisenstein series \cite{CD2}. 
Generalizing Cohen kernel, the arithmetic results of $L$-values by Manin \cite{M} and Shimura \cite{Sh} could be recovered \cite{CD1, CD2}.

The purpose of this paper to state above results to the space of Hilbert modular forms by extending kernel functions introduced in \cite{CD1, CD2}.
More precisely,  a double Hilbert Eisenstein series is a kernel function of two $L$-values of a primitive form in terms of the Petersson scalar product. Also one can recover the arithmetic results \cite{Sh} of  $L$-values of Hilbert cusp forms by studying Cohen kernel over real quadratic fields.
Furthermore it turns out that the Rankin-Cohen bracket of two Hilbert Eisenstein series is the special value of a double Hilbert Eisenstein series.

\medskip
 
\noindent {\bf Acknowledgment}: 
The first author was supported by NRF2018R1A4A1023 590 and 
 NRF2017R1A2B2001807. The second author was partially supported by by HIT Youth Talent Start-Up Grant and Grant of Technology Division of Harbin (RC2016XK001001).
\\

We thank the referee  for the valuable remarks, which led to improvements in the paper.

\section{Notations and Main Theorems}
Throughout of this paper, for simplicity, 
 we only consider the space of Hilbert modular forms over real quadratic fields $F$ with narrow class number one on the full Hilbert modular group $\Gamma=\text{SL}_2(\mathcal{O} ).$ 

\subsection{Notations}
Let $F$ be a real quadratic field with narrow class number equal to $1$. Let $D$, $\mathcal{O}$ and $\mathfrak d$ be the fundamental discriminant, the ring of integers and the different of $F$ respectively. 
Let $\textrm{N}$ and $\textrm{Tr}$ be the norm and the trace on $F$， 
defined by $\textrm{N}(a)=a a', \textrm{Tr}(a)=a+a'$ with 
$a'$ the algebraic conjugate of $a\in F$. 
We denote $a\gg 0$ for $a\in F$ if $a$ is totally positive, that is $a>0$ and $a'>0$. For $B \subset F$, let $B_+$ denote the subset of totally positive elements in $B$. So $\mathcal O_+$ and $\mathcal O^\times_+$ denote the set of totally positive integers and the set of totally positive units respectively.

For a $2\times 2$ matrix $\mathbf{\gamma}$ in $GL_2^+(F)$,  we usually denote its entries by
${\mathbf {\gamma}}=
\begin{smallmatrix}\begin{pmatrix}
 a_{\mathbf {\gamma}} &b_{\mathbf {\gamma}}\\
 c_{\mathbf {\gamma}} & d_{\mathbf {\gamma}}
 \end{pmatrix}
 \end{smallmatrix}$ and  $\gamma'=
 \begin{smallmatrix}\begin{pmatrix}
 a_{\mathbf {\gamma}}' &b_{\mathbf {\gamma}}'\\
 c_{\mathbf {\gamma}}' & d_{\mathbf {\gamma}}'
 \end{pmatrix}
  \end{smallmatrix}$.
The group $GL_2^+(F)$ acts on two copies of the complex upper half plane $\mathbb{H}^2$
by $\gamma z:=(\gamma z_1, \gamma' z_2)=(\frac{a_{\gamma}z_1+b_{\gamma}}{c_{\gamma}z_1+d_{\gamma}}, \frac{a_{\gamma}'z_2+b_{\gamma}'}{c_{\gamma}'
z_2+d_{\gamma}'})
$ as linear fractional transformations  for   all $  \gamma\in GL_2^+(F)$ and $ z=(z_1, z_2)\in \mathbb{H}^2$.

Let $\Gamma=\textrm{SL}_2(\mathcal O)$ be the modular group of $2\times 2$ matrices with determinant equal to one over $\mathcal O$. Denote $\Gamma_\infty$ the subgroup of upper-triangular elements and $\Gamma_\infty^+$ the subgroup of elements with totally positive diagonal entries in $\Gamma_\infty$.
Let $A$ denote the subgroup of diagonal elements in $\Gamma_\infty^+$, so $A=\{\text{diag}(\varepsilon,\varepsilon^{-1})\colon \varepsilon\in\mathcal O^\times_+\}$.
 Throughout the note, we employ the standard multi-index notation. In particular, for $\gamma\in \text{GL}_2^+(F)$, $z=(z_1,z_2) \in\mathbb{H}^2$ and $k\in\mathbb{Z}$, we denote $\mathbf{1}=(1,1)$, 
 $(\gamma z)^{k\mathbf{1}}=N(\gamma z)^k= (\gamma z_1)( \gamma' z_2)$, $ 
|z|=(|z_1|, |z_2|)$, $|z|^{k\mathbf{1}}=|z_1|^k|z_2|^k $   
and the automorphic factor by
\[j(\gamma,z)^{k\mathbf{1}}=N(j(\gamma,z))^k=j(\gamma,z_1)^kj(\gamma',z_2)^k=(c_\gamma z_1+d_\gamma)^k(c_\gamma'z_2+d_\gamma')^k.\]

For any function $f$ on $\mathbb H^2$ and $\gamma \in \textrm{GL}_2^+(F)$, define the slash operator by 
\[ (f|_k\gamma)(z)=N(\det(\gamma))^{\frac{k}{2}}N(j(\gamma,z))^{-k}f(\gamma z).\]
A Hilbert modular form of (parallel) weight $k$ for $\Gamma$ is a holomorphic function $f$ on $\mathbb{H}^2$ such that $f|_k\gamma=f$ for any $\gamma\in\Gamma$. Then $f$ has the following Fourier expansion
\[f(z)=a_f(0)+\sum_{\alpha\in\mathfrak d^{-1}_+}a_f(\alpha) e^{2\pi i\text{tr}(\alpha z)}.\] If $a_f(0)=0$, we call $f$ a Hilbert cusp form.
For a Hilbert cusp form $f$ and a Hilbert modular form $g$ of weight $k$ on $\Gamma$,
their Petersson scalar product is defined by
$$\langle f\, , g\rangle  :=\int_{\Gamma \backslash \mathbb{H}^2} f(z)\overline{g(z)}d\mu
=\int_{\mathcal{F}}    f(z)\overline{g(z)}d\mu,$$
where $\mathcal{F}$ is a fundamental domain of $\Gamma$ on $\mathbb{H}^2$ and \[d\mu= (y_1y_2)^{-2}dx_1dx_2dy_1dy_2 =N(y)^{-2}N(dx)N(dy).\]
Here  $z=x+iy, $ $Re(z)=x=(x_1,x_2)$ and  $Im(z)=y=(y_1, y_2)$.

Note that this ``unnormalized" Petersson inner product is different from Shimura's \cite{Sh}.
For a Hilbert cusp form $f$ of weight $k$ for $\Gamma$, define the associated L-function by
\[L(f,s)=\sum_{\alpha\in \mathfrak d^{-1}_+/\mathcal O_+^\times}a_f(\alpha) N(\alpha\mathfrak d)^{-s}=\sum_\mathfrak{a}a_f(\mathfrak a)N(\mathfrak a)^{-s},\]
where $a_f(\mathfrak a):=a_f(\alpha)$ for $\alpha\mathfrak d=\mathfrak a$.
It is known \cite{G}
 that the complete $L$-function satisfies
 \[\Lambda(f,s):= D^s(2\pi)^{-2s}\Gamma(s)^2L(f,s)=(-1)^k\Lambda(f,k-s) \]
and has an analytic continuation to the entire $\mathbb{C}$.

Next we recall the theory of Hecke operators on spaces of Hilbert modular forms. For each nonzero integral ideal $\mathfrak{n}$ of $\mathcal{O}$, let $M_\mathfrak{n}$ be the set of $2\times 2$ matrices $\gamma$ over $\mathcal O$ such that $\det(\gamma)\gg 0$ and $(\det(\gamma))=\mathfrak{n}$. Moreover, let $Z\cong \mathcal{O}^\times$ denote the $2\times 2$ scalar matrices with diagonal entries in $\mathcal O^\times$.
The $ \mathfrak{n}$-th Hecke operator $T_\mathfrak{n}$ on $S_k(\Gamma)$, the space of cusp forms for $\Gamma$  of parallel weight-$k$, is defined as
\[T_\mathfrak{n}(f(z))=N(\mathfrak n)^{ \frac{k}{2}-1}\sum_{\gamma\in Z\Gamma\backslash M_\mathfrak{n}} f|_k\gamma(z).\] The operators $T_\mathfrak{n}$ are self-adjoint with respect to the Petersson inner product and generate a commutative algebra. It follows that there exists a basis  $\mathcal H_k$, consisting of normalized cuspidal Hecke eigenforms,
of $S_k(\Gamma)$. We call elements in $\mathcal H_k$ ``primitive forms".
Here $f$ is normalized if the Fourier coefficient $a_f(\mathcal O)=1$ or equivalently if $\mathfrak d^{-1}=(\alpha)$ with $\alpha\gg 0$, then $a_f(\alpha)=1$. Therefore, for $f\in \mathcal H_k$, $T_\mathfrak{n}f=a_f(\mathfrak n)f$, so $a_f(\mathfrak n)$ is real. For details see Section 1.15 of \cite{G}.

\bigskip

\subsection{Main Theorems}
Fix $k\in \mathbb{Z}.$ We define the \emph{Cohen kernel} $\mathcal C_k^{Hil}(z;s) $ on $\mathbb{H}^2\times \mathbb{C}$ by
\begin{eqnarray}
\mathcal C_k^{Hil}(z;s)= \frac{1}{2} c_{k, s, D}^{-2} \sum_{\gamma\in A \backslash\Gamma}(\gamma z)^{-s\mathbf{1}}j(\gamma,z)^{-k\mathbf{1}}, 
\end{eqnarray}
with
\[c_{k, s,D}= \frac{D^{\frac{k-1}{2}} 2^{2-k}\pi \Gamma(k-1)}{e^{\frac{\pi i s}{2}} \Gamma(s) \Gamma(k-s)}\]
and $
A=\{\textrm{diag}(\varepsilon,\varepsilon^{-1})\colon
 \varepsilon\in\mathcal O_+^\times  \}.$ 
Note that if $k$ is odd, this definition gives zero function.
 
\medskip

\begin{thm}\label{cohen-analytic}({\bf Cohen kernel}) Let $k\geq 4$ be even.  Then the following hold:
\begin{enumerate} 
\item
$\mathcal C_k^{Hil}(z;s)$ converges absolutely and uniformly on all compact subsets in the region given by
\[1<\re(s)<k-1,\quad z\in\mathbb H^2.\] 

\item For each $s\in\mathbb{C}$,
\[\mathcal C_k^{Hil}(z;s)=\sum_{f\in\mathcal H_k}\frac{\Lambda(f,k-s)} {\langle f,f\rangle}f(z), \]
 where $\mathcal H_k$ is the set of  primitive forms s of weight $k$ on $\Gamma.$

\item  $\mathcal C_k^{Hil}(z;s)$
can be analytically continued to the whole $s$-plane and for each $s\in\mathbb{C}$, $\mathcal C_k^{Hil}(z;s)$ is a cusp form for $\Gamma$ of weight $k$ in $z$. 
\end{enumerate}
\end{thm}
\bigskip

 \medskip

Next we define the \emph{double Eisenstein series} as follows: 
for $s,w\in\mathbb C, z\in\mathbb H^2$ and even integer $k\geq 6$, 

\begin{eqnarray*}\label{Doublehilbert}
E_{s,k-s}^{Hil}(z;w)=\sum_{\gamma,\delta\in\Gamma_\infty^+\backslash\Gamma, c_{\gamma\delta^{-1}}\gg 0}(c_{\gamma\delta^{-1}})^{(w-1)\mathbf{1}}\left(\frac{j(\gamma,z)}{j(\delta,z)}\right)^{-s\mathbf{1}}j(\delta,z)^{-k\mathbf{1}},
\end{eqnarray*}
and a completed double Eisenstein series by
$$E^{*, Hil}_{s,k-s}(z;w)=2 \alpha_{k,s,w,D}\cdot  E_{s,k-s}^{Hil}(z;w) $$ 
  with 
\begin{eqnarray*}\label{con}
 \alpha_{k,s,w,D}&:=&
 D^{k-w} \zeta_F(1-w+s)\zeta_F(1-w+k-s) \nonumber \\
&&\times
\biggr( e^{\frac{is\pi}{2}}  (2\pi)^{ w-k-1} 2^{k-2}  \frac{\Gamma(s) \Gamma(k-s) \Gamma(k-w) }{\Gamma(k-1) }\biggl)^2.
 \end{eqnarray*}

Then we have the following:
\begin{thm}\label{main1}({\bf double Eisenstein series})  Let $k\geq 6$ be even. 
\begin{enumerate}
 
\item
$E_{s,k-s}^{Hil}(z;w)$ converges absolutely and uniformly on compact subsets in the region $\mathcal R$ of points $(z,(s,w))$ in $\mathbb H^2\times\mathbb C^2$ subject to
\[
2<\re(s)<k-2,\, \re(w)<\min\{\re(s)- 1, k- 1-\re(s)\}.\]

\item  
$E^{*, Hil}_{s,k-s}(z;w)$ has an analytic continuation to all $s,w\in \mathbb{C}$ and 
  is a Hilbert cusp form of weight $k$ on $\Gamma$
as a function in $z.$ 

\item  $$E^{*, Hil}_{s,k-s}(\cdot ; w)=\sum_{f\in \mathcal{H}_k}\frac{\Lambda(f,s)\Lambda(f,w)}{\langle f, f\rangle} f,$$ where $\mathcal{H}_k$ is   the   set of primitive forms of   weight $k.$

\item For $f\in \mathcal{H}_k$,  $\langle E^{*, Hil}_{s,k-s}(\cdot ; w), \, f\rangle=\Lambda(f,s)\Lambda(f,w), \text{ for all } s, w\in \mathbb{C}.$

\item $E^{*, Hil}_{s,k-s}(z;w)$ satisfies functional equations:
\[E^{*, Hil}_{s,k-s}(z;w)=E^{*, Hil}_{w,k-w}(z;s),\quad E^{*, Hil}_{k-s,s}(z;w)=E^{*, Hil}_{s,k-s}(z;w).\]
\end{enumerate}
\end{thm}
\medskip
 
 The following gives a relation between Rankin-Cohen brackets and a double Eisenstein series.
Rankin-Cohen brackets on spaces of Hilbert modular forms have been studied in \cite{CO}. Let us recall the definition of Rankin-Cohen brackets:
 for each $j=1,2,$ let $f_j:\mathbb{H}^2\rightarrow \mathbb{C}$ be holomorphic, 
  $k_j\in \mathbb{N}$  and $\ell=(\ell_1, \ell_2), \nu=(\nu_1,\nu_2)\in \mathbb{Z}_{\geq0}^2.$
Define the $\mathbb{\nu}$-th Rankin-Cohen bracket 
$$[f_1,f_2]^{Hil}_{\nu}=\sum_{\begin{array}{cc}0\leq \ell_j\leq \nu_j, j=1,2\end{array}
}(-1)^{\ell_1+\ell_2}\sm  k_1\mathbf{1}+\nu-\mathbf{1} \\ \nu-\ell\esm \sm   k_2\mathbf{1}+\nu-\mathbf{1}  \\ \ell \esm 
f_1^{(\ell)}f_2^{(\nu-\ell)}.$$
Here $f^{(\ell)}(z)=
(\frac{\partial^{\ell_1+\ell_2}}{
{\partial z_1}^{\ell_1}{\partial z_{2}}^{\ell_2}
}
f)(z)$ and
$ \sm   k\mathbf{1}+\nu-\mathbf{1} \\ \nu-\ell\esm = \sm  {k} +\nu_1-1\\ \nu_1-\ell\esm \sm  {k} +\nu_2-1\\ \nu_2-\ell\esm$.
\medskip

In the following, we only need parallel $\nu$, that is $\nu_1=\nu_2$.
\begin{thm}\label{RC} ({\bf  Rankin-Cohen brackets and a double Eisenstein series})
For $\nu\in \mathbb{Z}_{\geq 0}$ and $ k_j\in 2\mathbb{N}, j=1,2$, we have 
$$\left(\frac{\Gamma( {k_1}) \Gamma(\nu+1) }{\Gamma( {k_1+\nu})}\right)^2
[E_{ {k_1}}, E_{ {k_2}}]^{Hil}_{(\nu, \nu)}
=   4 \left(\frac{\Gamma( {k_2+\nu})  }{\Gamma( {k_2})}\right)^{ 2 } E^{ Hil}_{ {k_1}+\nu,  {k_2}+\nu}(z; \nu+1), $$ where
$E_{ {k }}(z) $ is the usual Hilbert Eisenstein series of weight $k$ on $\Gamma$  defined by 
$$E_{ {k }}(z) :=\sum_{\gamma\in \Gamma_{\infty}^+ \backslash \Gamma} j(\gamma, z)^{-k\mathbf{1}}.$$
\end{thm}
  
\begin{rmk}
\begin{enumerate}
\item  
Cohen kernel (see \cite{C} and \cite{KZ}) is an elliptic cusp form $R_n$ of weight $2k$ on $\text{SL}_2(\mathbb{Z})$ characterized by, for each $  0\leq n\leq 2k-2,$
$$\langle f, \, R_n\rangle =n!(2\pi)^{-n-1} L(f, n+1), \text{ for all } f\in S_{2k}(\text{SL}_2(\mathbb{Z})).$$
  Diamantis and O'Sullivan in \cite{CD1} generalized  Cohen kernel $\mathcal{C}^{ell}_k(\tau, s) $   to get
$$\langle f, \, \mathcal{C}^{ell}_k(\tau, s) \rangle = \Gamma(s)(2\pi)^{s} L(f, s),
  s\in \mathbb{C}.$$

\item
Double Eisenstein series was introduced  and studied in  \cite{CD1, CD2} as a kernel yielding products of the periods of an elliptic Hecke eigenform at critical values as well as producing products of $L$-functions for Maass cusp forms.

\end{enumerate}
\end{rmk}
 
\medskip

In the following theorem, we recover Shimura's result on the algebraicity of critical values of $L(f, s)$ (Theorem 4.3 of \cite{Sh}).
For a primitive   form $f$ of even weight $k$, let $\mathbb{Q}(f)$ denote the number field generated by the Fourier coefficients of $f$ over $\mathbb{Q}$.

\begin{thm}\label{ra1}  ({\bf rationality}) 
Let $f$ be a primitive form of even weight $k\geq 6$ for $\Gamma$.  
Then there exist complex numbers $\omega_{\pm}(f) $ with $\langle f,\,f\rangle =\omega_{+}(f) \omega_{-}(f)$ such that for even $m$ and odd $\ell$ with $1\leq m,\ell\leq k-1$,

\begin{enumerate}
\item[(1)]
 \[\frac{\Lambda(f,m)}{w_{+}(f)}, \frac{\Lambda(f,\ell)}{w_{-}(f)}  \in \mathbb{Q}(f),\] 

\item[(2)]
 for each $\sigma\in\text{Gal}(\bar{\mathbb{Q}}/\mathbb{Q})$, 
\[\left(\frac{\Lambda(f,m)}{\omega_+(f)}\right)^\sigma=\frac{\Lambda(f^\sigma,m)}{\omega_+(f^\sigma)},\quad \left(\frac{\Lambda(f,\ell)}{\omega_-(f)}\right)^\sigma=\frac{\Lambda(f^\sigma,\ell)}{\omega_-(f^\sigma)}.\] 
\end{enumerate}
\end{thm}

\medskip
\begin{rmk}
\begin{enumerate}
\item
The above theorem is an analogous result of that for elliptic modular forms proved in \cite{KZ} (Theorem in page 202).  We can also extend the rationality easily to arbitrary L-values as did in Theorem 8.3 of \cite{CD2}.
\item The above theorem  is a special case of Shimura theorem (Theorem 4.3 in \cite {Sh}) by taking    $n=2$, $\psi=1$, and $k_1=k_2=k.$ 
\end{enumerate}
\end{rmk}

\section{Proofs}

We need the following multi variable Lipschitz summation formula.

\begin{lem}\label{lip} ({\bf multi-variable Lipschitz summation formula})
Assume that $\im (s)>2$. For $z\in \mathbb{H}^2$, 
\[\sum_{x\in\mathcal O}(z+x)^{-s\mathbf{1}}
=\frac{(2\pi)^{2s}}{e^{\pi is}\Gamma(s)^{2}D^{1/2}}\sum_{\xi\in\mathfrak d^{-1}_+}N(\xi)^{s-1}\exp(2\pi i\textrm{Tr}(\xi z)),\] 
\end{lem}

\begin{proof}  By the multi-index notation, 
\[\sum_{x\in\mathcal O}(z+x)^{-s\mathbf{1}}= \sum_{x\in\mathcal O}
N(z+x)^{-s}=\sum_{x\in\mathcal O}
(z_1+x)^{-s}(z_2+x')^{-s}.\] 
Following \cite{KR},   define   \[f(x)=N(x)^{s-1}\exp(2\pi i\tr(xz))\] for $x=(x_1,x_2)\gg 0$ and $0$ otherwise, so for $  \im (s)>2$ and $z\in\mathbb{H}^2$, $f$ is clearly $L^1$ on the quadratic space $V=\mathbb{R}^2$ with the trace form. The computation of \cite[Theorem 1]{KR} shows that the Fourier transform $\hat{f}(w)$ is given by
\[\hat{f}(w)=\frac{\Gamma(s)^2}{(-2\pi i)^{2s}}(z+w)^{-s\mathbf{1}}, \quad w\in\mathbb{R}^2.\] It is clear that for $x\in\mathbb{R}^2$, 
\[|f(x)|+|\hat f(x)|\ll (1+  ||x|| )^{-2-\delta}\] 
for any positive $\delta$,   where $||\cdot||$ is the Euclidean norm. Therefore, we may apply the Poisson summation formula (see page 252 of \cite{St}), and for  a general lattice $M$ in $V$ with integral dual lattice $M^\vee$, the Poisson summation formula reads \[\sum_{\alpha\in M}f(\alpha)=\sqrt{|M/M^{\vee}|}\sum_{\alpha\in M^{\vee}}\hat{f}(\alpha).\] Now set $M= \mathfrak{d}^{-1}$, then $M^\vee=\mathcal O$,   $|M/M^\vee|=D$    and the Lipschitz summation formula follows easily. 
\end{proof}
\medskip

\noindent Now we prove Theorem \ref{cohen-analytic} about Cohen kernel.

\noindent\emph{Proof of Theorem   \ref{cohen-analytic} }:
To show the convergence,  we follow the treatment of Section 1.15 in \cite{G}. 
Firstly, we prove the uniform absolute convergence on compact subsets, using the fact that $L^1$-convergence implies uniform convergence on compact subset for series of holomorphic functions (See Lemma on Page 52 of \cite{G}). It suffices to treat the case for $z$ in a small neighborhood $U$ such that $\overline U$ is compact,  $ N(\im z) > X^{-1}$  and $ N(\im\gamma z) < X$ for any $\gamma\in\Gamma$ and $z\in U$ for fixed big $X>0$. Note that this essentially picks a Siegel set where $\overline{U}$ lives.   In this case, we only have to prove that 
\[\int_{\Gamma\backslash\mathbb H^2_X}\sum_{\gamma\in A\backslash\Gamma}|N(j(\gamma,z))|^{-k}|\gamma z|^{-\sigma\mathbf{1}}d\mu(z)<\infty\]
where $\sigma=\re(s)$ and  $\mathbb{H}^2_X$ is the subset of $z$  with  $ N(\im z)<X$ in $\mathbb{H}^2$.  Here we denote $|z|=(|z_1|,|z_2|)$ and employ the multi-index notation.
The left-hand side is bounded by
\begin{align*}
\leq & X^{\frac{k}{2}}\int_{\Gamma\backslash\mathbb H^2_X}\sum_{\gamma\in A\backslash\Gamma}(\im \gamma z)^{\frac{k}{2}\mathbf{1}}|\gamma z|^{-\sigma\mathbf{1}}d\mu(z)\\
\ll & \sum_{\gamma\in A\backslash\Gamma}\int_{\Gamma\backslash\mathbb H^2_X}(\im \gamma z)^{\frac{k}{2}\mathbf{1}}|\gamma z|^{-\sigma\mathbf{1}}d\mu(z)\\
=& \int_{A\backslash\mathbb H^2_X}(\im z)^{\frac{k}{2}\mathbf{1}}|z|^{-\sigma\mathbf{1}}d\mu(z).
\end{align*}
The space $A\backslash \mathbb H^2_X$ can be viewed as a subspace of 
\[\{(z_1,z_2)\colon   N( \im z)  <X,  Y^{-1}\leq y_1/y_2\leq Y   \}\] for some positive $Y$ ($Y$ can be chosen as the smallest totally positive unit bigger than $1$). Moreover,
 that $N(\im (-1/z))<X$ implies $N(|z|^2) >X^{-1} N(\im z)$.  For $1<r<\sigma<k-1$, the last quantity is equal to
\begin{align*}
& \int_{A\backslash\mathbb H^2_X}( N( \im z) )^{\frac{k}{2}}|Nz|^{-r-(\sigma-r)}d\mu(z)\\
\ll & \int_{A\backslash\mathbb H^2_X}( N( \im z))^{\frac{k-\sigma+r}{2}}|Nz|^{-r}d\mu(z)\\
\ll & \int_{y_1y_2<X, Y^{-1}<y_1/y_2<Y}( N( \im z)  )^{\frac{k-\sigma+r}{2}}( N( \im z)  )^{1-r}\frac{dy_1dy_2}{(y_1y_2)^2} \\
= & \int_{y_1y_2<X, Y^{-1}<y_1/y_2<Y} ( N( \im z))^{\frac{k-\sigma-r}{2}-1}dy_1dy_2<\infty
\end{align*}
 where in the third line we applied Equation (5.8) of \cite{CD1} for the integration on $x$. This is part (1).

For part (2), first note that the absolutely uniformly convergence implies that ${\mathcal C}^{Hil}_k (z;s)$ converges to a Hilbert modular  form in the strip $2<\sigma<k-1$ since ${\mathcal C}^{Hil}_k (z;s)$ is $\Gamma$-invariant with a proper automorphic factor. Secondly, we write
\begin{eqnarray*}
&& 2c_{k,s,D}^2 \cdot {\mathcal C}^{Hil}_k (z;s) =\sum_{\alpha\in A\backslash \Gamma_\infty^+}\sum_{\gamma\in \Gamma_\infty^+\backslash\Gamma} j(\alpha\gamma,z)^{-k\mathbf{1}}(\alpha\gamma z)^{-s\mathbf{1}}\\
&&=\sum_{\gamma\in \Gamma_\infty^+\backslash\Gamma} j(\gamma,z)^{-k\mathbf{1}}\sum_{\alpha\in A\backslash \Gamma_\infty^+}(\alpha\gamma z)^{-s\mathbf{1}} 
 =\sum_{\gamma\in \Gamma_\infty^+\backslash\Gamma} j(\gamma,z)^{-k\mathbf{1}}\sum_{x\in\mathcal O}(\gamma z+x)^{-s\mathbf{1}}.
\end{eqnarray*}
 Applying the Lipschitz summation formula in Lemma \ref{lip} with $2<\sigma<k-1$, we have
\begin{align*}
2c_{k,s,D}^2\cdot {\mathcal C}^{Hil}_k(z;s)
&=\frac{(2\pi)^{2s}}{e^{\pi is}\Gamma(s)^{2}D^{1/2}}\sum_{\gamma\in \Gamma_\infty^+\backslash\Gamma} j(\gamma,z)^{-k\mathbf{1}}\sum_{\xi\in\mathfrak d^{-1}_+}(N\xi)^{s-1}\exp(2\pi i\textrm{Tr}(\xi\gamma z))\\
&=\frac{(2\pi)^{2s}}{e^{\pi is}\Gamma(s)^{2}D^{1/2}}\sum_{\xi\in\mathfrak d^{-1}_+}(N\xi)^{s-1}\sum_{\gamma\in \Gamma_\infty^+\backslash\Gamma} j(\gamma,z)^{-k\mathbf{1}}\exp(2\pi i\textrm{Tr}(\xi\gamma z))\\
&=\frac{(2\pi)^{2s}}{e^{\pi is}\Gamma(s)^{2}D^{1/2}}\sum_{\xi\in\mathfrak d^{-1}_+/(\mathcal O_+^\times)^2}(N\xi)^{s-1}\\
&\qquad\times\sum_{u\in\mathcal O^\times_+}\sum_{\gamma\in \Gamma_\infty^+\backslash\Gamma} j(\gamma,z)^{-k\mathbf{1}}\exp(2\pi i\textrm{Tr}(u^2\xi\gamma z))\\
&=\frac{(2\pi)^{2s}}{e^{\pi is}\Gamma(s)^{2}D^{1/2}}\sum_{\xi\in\mathfrak d^{-1}_+/(\mathcal O_+^\times)^2}(N\xi)^{s-1}\sum_{\gamma\in U\backslash\Gamma} j(\gamma,z)^{-k\mathbf{1}}\exp(2\pi i\textrm{Tr}(\xi\gamma z)),
\end{align*}
where $U$ is the subgroup of elements of the form $\begin{pmatrix}
1&x\\0&1
\end{pmatrix}$ in $\Gamma$.
On the other hand, recall
 the $\xi$-th Poincar\'e series \cite{G}
$$P_k(z;\xi)=\sum_{\gamma\in U\backslash\Gamma} 
j(\gamma,z)^{-k\mathbf{1}}\exp(2\pi i\textrm{Tr}(\xi\gamma z))$$
 and that it is a cusp form with
$$ P_k(z;\xi)=\frac{\Gamma(k-1)^2D^{1/2}}{(4\pi)^{2k-2}N(\xi)^{k-1}}\sum_{f\in\mathcal H_k}\frac{\bar{a}_f(\xi)f}{\langle f,f\rangle}.$$
We see that up to a constant factor (depending on $s$) ${\mathcal C}^{Hil}_k(z;s)$ is equal to 
\begin{align*}
&\sum_{\xi\in\mathfrak d^{-1}_+/(\mathcal O_+^\times)^2}(N\xi)^{s-1}(N\xi)^{1-k}\sum_{f\in\mathcal H_k}\frac{\bar{a}_f(\xi)f(z)}{\langle f,f\rangle}\\
=&\sum_{\xi\in\mathfrak d^{-1}_+/(\mathcal O_+^\times)^2}(N\xi)^{s-k}\sum_{f\in\mathcal H_k}\frac{\bar{a}_f(\xi)f(z)}{\langle f,f\rangle}\\
=&2\sum_{\xi\in\mathfrak d^{-1}_+/\mathcal O_+^\times}(N\xi)^{s-k}\sum_{f\in\mathcal H_k}\frac{\bar{a}_f(\xi)f(z)}{\langle f,f\rangle}\\
=&2D^{k-s}\sum_{f\in\mathcal H_k}\frac{f(z)}{\langle f,f\rangle}\sum_{\xi\in\mathfrak d^{-1}_+/\mathcal O_+^\times}(N\xi\mathfrak{d})^{s-k}\bar{a}_f(\xi)\\
=&2D^{k-s}\sum_{f\in\mathcal H_k}\frac{f(z)L(f,k-s)}{\langle f,f\rangle},
\end{align*} where we used the fact that $a_f(\xi)$ is real.
Putting everything together, we see that
\[2c_{k,s,D}^2\cdot {\mathcal C}^{Hil}_k (z;s)=\frac{2^{5-2k}  \pi ^2\Gamma(k-1)^2}{e^{\pi is}\Gamma(s)^2\Gamma(k-s)^2}\sum_{f\in\mathcal H_k} \frac{\Lambda(f,k-s)f(z)}{\langle f,f\rangle}\]
It follows that $ {\mathcal C}^{Hil}_k (z;s)$ is cuspidal on the region $2<\sigma<k-1$, and that
\begin{eqnarray*}\label{pro}
{\mathcal C}^{Hil}_k(z;s)=\sum_{f\in\mathcal H_k}\frac{\Lambda(f,k-s)f(z)}{\langle f,f\rangle}.
\end{eqnarray*}

For part (3):  The expression of $ {\mathcal C}^{Hil}_k (z;s)$ in part (2) gives the analytic continuation to $s\in\mathbb{C}$ and that for each $s\in\mathbb{C}$, $ {\mathcal C}^{Hil}_k(z;s)$ is a cusp form. This completes the proof. $\Box$

\bigskip

Next, to prove Theorem \ref{main1} we first need  to show a connection between Cohen kernel and double Eisenstein   series, which is obtained in the following lemma: 

\begin{lem}\label{connection}
On the region $\mathcal R$, we have
\[\zeta_F(1-w+s)\zeta_F(1-w+k-s)E_{s,k-s}^{Hil}(z;w)=2c_{k,s,D}^2  \sum_{\mathfrak n}N(\mathfrak n)^{w-k}T_{\mathfrak{n}} \bigr( {\mathcal C}^{Hil}_k (z;s)\bigl), \]
with  $T_{\mathfrak{n}}$  the $\mathfrak{n}$-th Hecke operator 
and $\zeta_F(s)$ the Dedekind zeta function for $F$ defined as
\[\zeta_F(s)=\sum_{\mathfrak a}N(\mathfrak a)^{-s}=\sum_{a\in\mathcal O_+/\mathcal{O}_+^\times}N(a)^{-s},\]
where $\mathfrak a$ runs through all integral nonzero ideals. 
 \end{lem}
\begin{proof} 
On $\mathcal R$, the series expansions of the two $\zeta_F$-factors converge absolutely. Therefore, on $\mathcal R$, by sending $\gamma$ to $(c_
\gamma,d_\gamma)$, the left-hand side is equal to
\begin{eqnarray*} 
&& \zeta_F(1-w+s)\zeta_F(1-w+k-s)E_{s,k-s}^{Hil}(z;w)
 \nonumber\\
&=&
\sum_{u,\tilde{u}}N(u)^{w-1-s}N(\tilde{u})^{w+s-1-k}\sum_{(c,d),(\tilde{c},\tilde{d})} (c\tilde{d}-d\tilde{c})^{(w-1)\mathbf{1}}\left(\frac{cz+d}{\tilde{c}z+\tilde{d}}\right)^{-s\mathbf{1}}(\tilde{c}z+\tilde{d})^{-k\mathbf{1}},
\end{eqnarray*}
where $u,\tilde{u}\in\mathcal O_+^\times\backslash\mathcal O_+$ and $(c,d),(\tilde{c},\tilde{d})\in \mathcal{O}_+^\times\backslash\mathcal O^2$ such that $\mathcal Oc+\mathcal Od=\mathcal O\tilde{c}+\mathcal O\tilde{d}=\mathcal O$ and $c\tilde{d}-d\tilde{c}\gg 0$. Combining the two summations, we have
\begin{align*}
\sum_{\mathfrak a,\tilde{\mathfrak a}}\sum_{(c,d),(\tilde{c},\tilde{d})} (c\tilde{d}-d\tilde{c})^{(w-1)\mathbf{1}}\left(\frac{cz+d}{\tilde{c}z+\tilde{d}}\right)^{-s\mathbf{1}}(\tilde{c}z+\tilde{d})^{-k\mathbf{1}},
\end{align*}
where this time $\mathfrak a,\tilde{\mathfrak a}$ are over all nonzero integral ideals and the inner summation is over $(c,d),(\tilde{c},\tilde{d})\in \mathcal{O}_+^\times\backslash\mathcal O^2$ such that $\mathcal Oc+\mathcal Od=\mathfrak a$, $\mathcal O\tilde{c}+\mathcal O\tilde{d}=\tilde{\mathfrak a}$ and $c\tilde{d}-d\tilde{c}\gg 0$. Then we can remove the summation over $\mathfrak a,\tilde{\mathfrak a}$ and it equals to
\begin{align*}
&   \sum_{(c,d),(\tilde{c},\tilde{d})}  (c\tilde{d}-d\tilde{c})^{(w-1)\mathbf{1}}\left(\frac{cz+d}{\tilde{c}z+\tilde{d}}\right)^{-s\mathbf{1}}(\tilde{c}z+\tilde{d})^{-k\mathbf{1}}\\
=&\sum_{\mathfrak n}\sum_{(c,d),(\tilde{c},\tilde{d})} (c\tilde{d}-d\tilde{c})^{(w-1)\mathbf{1}}\left(\frac{cz+d}{\tilde{c}z+\tilde{d}}\right)^{-s\mathbf{1}}(\tilde{c}z+\tilde{d})^{-k\mathbf{1}},
\end{align*}
where $\mathfrak n$ is over all nonzero integral ideals and the inner summation is over over $(c,d),(\tilde{c},\tilde{d})\in \mathcal{O}_+^\times\backslash\mathcal O^2$ such that $c\tilde{d}-d\tilde{c}\gg 0$ and $(c\tilde{d}-d\tilde{c})=\mathfrak n$.   Note that the two summations over $(c,d),(\tilde{c},\tilde{d})$ in the preceding equation have different ranges. 

Let $\tilde{A}$ denote the group of diagonal $2\times 2$ matrices with entries in $\mathcal O_+^\times$, so clearly $\tilde A\subset Z\Gamma$ and $\tilde A\backslash Z\Gamma\cong A\backslash\Gamma$.  Note that the inner summation set is mapped bijectively to $\tilde A\backslash M_\mathfrak{n}$ via
\[((c,d),(\tilde{c},\tilde{d}))\mapsto \begin{pmatrix}
c&d\\\tilde{c}&\tilde{d}
\end{pmatrix}.\]
Therefore, above expression is equal to
\begin{align*}
&\sum_{\mathfrak n}\sum_{\gamma\in \tilde A\backslash M_\mathfrak{n}} (\det(\gamma))^{(w-1)\mathbf{1}}(\gamma z)^{-s}j(\gamma,z)^{-k\mathbf{1}}\\
=& \sum_{\mathfrak n}\sum_{\gamma\in Z\Gamma\backslash M_\mathfrak{n}}\sum_{\beta\in \tilde A\backslash Z\Gamma} (\det(\beta\gamma))^{(w-1)\mathbf{1}}(\beta\gamma z)^{-s\mathbf{1}}j(\beta\gamma,z)^{-k\mathbf{1}}\\
=& \sum_{\mathfrak n}\sum_{\gamma\in Z\Gamma\backslash M_\mathfrak{n}}\sum_{\beta\in A\backslash \Gamma} (\det(\gamma))^{(w-1)\mathbf{1}}(\beta\gamma z)^{-s\mathbf{1}}j(\beta\gamma,z)^{-k\mathbf{1}}\\
=& 2c_{k,s,D}^2\cdot \sum_{\mathfrak n}N(\mathfrak{n})^{-\frac{k}{2}+w-1 }\sum_{\gamma\in Z\Gamma\backslash M_\mathfrak{n}} {\mathcal C}^{Hil}_k (z;s)|_k\gamma\\
=& 2c_{k,s,D}^2\cdot  \sum_{\mathfrak n}N(\mathfrak{n})^{  w-k  }T_\mathfrak{n}\bigl( {\mathcal C}^{Hil}_k (z;s)\bigr),
\end{align*}
which is the right-hand side.
\end{proof}
   
Using the preceding lemma we prove the following main theorem:

\noindent\emph{Proof of  Theorem \ref{main1}} :
For part (1),   apply the proof of Lemma 4.1 in \cite{CD2} for each component  and we have 
$$N(c_{\gamma\delta^{-1}})\leq N(\im(\gamma z))^{-1/2}N(\im(\delta z))^{-1/2},$$
for any $\gamma,\delta\in\Gamma$ with $c_{\gamma\delta^{-1}}\gg 0$.
Let $r=\max\{\re(w),1\}$. Since $[\Gamma_\infty:\Gamma_\infty^+]$ is finite, $E_{s,k-s}(z;w)$ is absolutely bounded up to a constant by
\begin{align*}
&\sum_{\gamma,\delta\in\Gamma_\infty^+\backslash\Gamma, c_{\gamma\delta^{-1}}\gg 0}(Nc_{\gamma\delta^{-1}})^{\re(w)-1}|Nj(\gamma,z)|^{-\re(s)}|Nj(\delta,z)|^{\re(s)-k}\\
\leq &\sum_{\gamma,\delta\in\Gamma_\infty^+\backslash\Gamma, c_{\gamma\delta^{-1}}\gg 0}N(\im(\gamma z))^{\frac{1-r}{2}}N(\im(\delta z))^{\frac{1-r}{2}}|Nj(\gamma,z)|^{-\re(s)}|Nj(\delta,z)|^{\re(s)-k}\\
\leq & \sum_{\gamma,\delta\in\Gamma_\infty^+\backslash\Gamma, c_{\gamma\delta^{-1}}\neq 0}N(\im(\gamma z))^{\frac{1-r}{2}}N(\im(\delta z))^{\frac{1-r}{2}}|Nj(\gamma,z)|^{-\re(s)}|Nj(\delta,z)|^{\re(s)-k}\\ 
\ll & \sum_{\gamma,\delta\in\Gamma_\infty\backslash\Gamma, c_{\gamma\delta^{-1}}\neq 0}N(\im(\gamma z))^{\frac{1-r}{2}}N(\im(\delta z))^{\frac{1-r}{2}}|Nj(\gamma,z)|^{-\re(s)}|Nj(\delta,z)|^{\re(s)-k}\\ 
\ll & N(y)^{-\frac{k}{2}}\sum_{\gamma,\delta\in\Gamma_\infty\backslash\Gamma, c_{\gamma\delta^{-1}}\neq 0}N(\im(\gamma z))^{\frac{\re(s)-  r +1}{2}}N(\im(\delta z))^{\frac{k-\re(s)- r  +1}{2}}\\ 
\ll & N(y)^{-\frac{k}{2}}\sum_{\gamma,\delta\in\Gamma_\infty\backslash\Gamma}N(\im(\gamma z))^{\frac{\re(s)-r +1}{2}}N(\im(\delta z))^{\frac{k-\re(s)- r +1}{2}},
\end{align*}
which is the product of two Eisenstein series whose absolute convergence is well-known (see, for example, 5.7 Lemma of Chapter I in \cite{F}). So absolute convergence follows if we have 
\[\frac{\re(s)- r +1}{2}> 1 \quad\text{and}\quad \frac{k-\re(s)- r+1}{2}  > 1.\] 
One sees easily that $E_{s,k-s}(z;w)$ transforms correctly under $\Gamma.$ 
 In the above estimate  
 \begin{eqnarray*}
&&
\sum_{\gamma,\delta\in\Gamma_\infty\backslash\Gamma, c_{\gamma\delta^{-1}}\neq 0}N(\im(\gamma z))^{\frac{\re(s)-  r +1}{2}}N(\im(\delta z))^{\frac{k-\re(s)- r  +1}{2}}
\\
&&=E\left(z, \frac{\re(s)-r+1}{2}\right) E\left(z, \frac{k-\re(s)-r+1}{2}\right) -E\left(z, \frac{k-2r+2}{2}\right),
\end{eqnarray*}
where 
$E(z,s):=\sum_{\gamma\in \Gamma_{\infty} \backslash \Gamma}
N(Im(\gamma z))^{-s} =N(y)^{s}+A(s)N(y)^{1-s}+o(1)$.  
By removing the highest terms $N(y)^{\frac{k}{2}-r+1}$ from the difference, the rest are all $o(N(y)^{\frac{k}{2}}). $   This shows that $E_{s,k-s}(z;w)\rightarrow 0$ as $N(y) \rightarrow\infty$, and hence proves part (2) that $E_{s,k-s}(z;w)$ is a cuspform since only one cusp exists. 

For part (3), by Theorem \ref{cohen-analytic} the Cohen kernel are cuspforms. By Lemma \ref{connection}, 
\begin{align*}
&\zeta_F(1-w+s)\zeta_F(1-w+k-s)   E_{s,k-s}^{Hil}(z;w)\\
=&2c_{k,s,D}^2\sum_{\mathfrak n}N(\mathfrak n)^{w-k}T_\mathfrak{n}\bigl( {\mathcal C}^{Hil}_k (z;s)\bigr)\\
=&2c_{k,s,D}^2\sum_{\mathfrak n}N(\mathfrak n)^{w-k}\sum_{f\in\mathcal H_k}\frac{\langle T_\mathfrak{n}{\mathcal C}^{Hil}_k (z;s),f\rangle}{\langle f,f\rangle}f(z)\\
=&2c_{k,s,D}^2\sum_{\mathfrak n}N(\mathfrak n)^{w-k}\sum_{f\in\mathcal H_k}\frac{\langle {\mathcal C}^{Hil}_k (z;s),T_\mathfrak{n} f\rangle}{\langle f,f\rangle}f(z)\\
=&2c_{k,s,D}^2 \sum_{\mathfrak n}N(\mathfrak n)^{w-k}\sum_{f\in\mathcal H_k} a_f(\mathfrak{n})
\frac{\langle {\mathcal C}^{Hil}_k (z;s),f\rangle}{\langle f,f\rangle}f(z),
\end{align*} since $\overline{a_f(\mathfrak{n})}=a_f(\mathfrak{n})$.
We have shown in Theorem \ref{cohen-analytic} that 
\begin{equation*}
2c_{k,s,D}^2\langle  {\mathcal C}^{Hil}_k (z;s),f\rangle=  \frac{2^{5-2k}\pi^2\Gamma(k-1)^2}{e^{\pi i s} \Gamma(s)^2\Gamma(k-s)^2}\Lambda(f,k-s),
\mbox{ for $f\in \mathcal{H}_k$}.
\end{equation*}
By defining
\[E^{^*,Hil}_{s,k-s}(z;w)=2\alpha_{k,s,w,D}
  E_{s,k-s}^{Hil}(z;w)\]
with
 $\alpha_{k,s,w,D}$  in (\ref{con}) and  using the result of   Theorem \ref{cohen-analytic},
 we obtain 
\[E^{*, Hil}_{s,k-s}(z;w)=\sum_{f\in\mathcal H_k}\frac{\Lambda(f,k-w)\Lambda(f,k-s)}{\langle f,f\rangle}f(z)
=\sum_{f\in\mathcal H_k}\frac{\Lambda(f,w)\Lambda(f,s)}{\langle f,f\rangle}f(z). \]

Part (4) follows easily from part (3) since $f$ is a primitive form. Finally,  by part (3), $E^{*, Hil}_{s,k-s}(z;w)$ has meromorphic continuation to all of $s,w\in \mathbb C$, and reflected from properties of $\Lambda(f, s)$, it satisfies functional equations
\[E^{*, Hil}_{s,k-s}(z;w)=E^{*, Hil}_{w,k-w}(z;s),\quad E^{*, Hil}_{k-s,s}(z;w)=E^{*, Hil}_{s,k-s}(z;w),\] 
proving part (5) and hence the whole theorem.
$\Box$
 
\medskip

Using the result about Rankin-Cohen brackets studied in \cite{CO}, we prove Theorem \ref{RC}:

\noindent\emph{Proof of Theorem  \ref{RC}}: 
One checks (from Proposition 1 in  \cite{CO})
\begin{eqnarray*}
&& 
\left(\frac{( {k_1}-1)!\nu!}{( {k_1}+\nu-1)!}\right)^2
[E_{  {k_1}}, E_{  {k_2}} ]^{Hil}_{(\nu, \nu ) }
 =
\sum_{\delta\in \Gamma_{\infty}\backslash \Gamma}  
j(\delta, z)^{- {k_1}\mathbf{1}}  E_{  {k_2}}^{(\nu)}|_{ {k_2}+2\nu}\delta.
\end{eqnarray*}
Since 
\[E_{  {k_2}}^{(\nu)}= \left(\frac{( {k_2}-1+\nu)!}{( {k_2}-1)!} \right)^2\sum_{\gamma\in  \Gamma^+_{\infty}\backslash\Gamma}
N(c_{\gamma})^{\nu} 
j(\gamma, z)^{-( {k_2}+\nu)\mathbf{1}} \]   by Lemma 1 in \cite{CO}, this in turn is equal to

\begin{align*}
&\left(
\frac{( {k_2}-1+\nu)!}{( {k_2}-1)!}\right)^2\sum_{\delta, \gamma \in \Gamma^+_{\infty}\backslash \Gamma}  
j(\delta, z)^{- {k_1}\mathbf{1}} 
N(c_{\gamma})^{\nu} 
j(\gamma, \delta(z))^{-( {k_2}+\nu)\mathbf{1}}
j(\delta, z)^{-( {k_2}+2\nu)\mathbf{1}}\\
=& \left(
\frac{( {k_2}-1+\nu)!}{( {k_2}-1)!}\right)^2
 \sum_{\delta, \gamma \in \Gamma_{\infty}^+\backslash \Gamma} N(c_{\gamma\delta^{-1}})^{\nu} 
j(\delta, z)^{- {(k_1+\nu)}\mathbf{1}} j(\gamma,z)^{- {(k_2+\nu)}\mathbf{1}}.
\end{align*}
In such a particular situation, we see easily that the summand is actually well-defined on $\Gamma_\infty\backslash\Gamma$. Denote $S$ the subset of $(\delta,\gamma)\in (\Gamma_{\infty}^+\backslash \Gamma)^2$ with $c_{\gamma\delta^{-1}}\neq 0$ and $S_{\pm,\pm}\subset S$ consists of elements whose $c_{\gamma\delta^{-1}}$ has the prescribed sign vector. In particular, $S_{+,+}$ consists of elements with $c_{\gamma\delta^{-1}}\gg 0$. It is obvious that the sums over these four subsets are all equal, since we may multiply on left by $\pm I$ and $\pm\text{diag} (\varepsilon_0,\varepsilon_0^{-1})\in\Gamma_\infty$ to adjust the signs; here $\varepsilon_0$ is the fundamental unit. That said, we have
\[\left(\frac{( {k_1}-1)!\nu!}{( {k_1}+\nu-1)!}\right)^2
[E_{  {k_1}}, E_{  {k_2}} ]^{Hil}_{(\nu, \nu ) }
 = 4\left(
\frac{( {k_2}-1+\nu)!}{( {k_2}-1)!}\right)^2E^{Hil}_{ {k_1} +\nu,  {k_2} +\nu}(z, \nu+1),\]
and it finishes the proof. $\Box$

\medskip

\noindent\emph{Proof of Theorem \ref{ra1}}: We follow the lines in Section 8A of \cite{CD2} and first prove that for even $m$ and odd $\ell$ with $1\leq m,\ell\leq k-1$, both of $E^{*, Hil}_{m, k-m}(z; k-1)$ and $E^{*, Hil}_{k-2, 2}(z; \ell)$ have rational Fourier coefficients. By the functional equations in Theorem \ref{main1}, 
\[E^{*, Hil}_{m, k-m}(z; k-1)=E^{*, Hil}_{m, k-m}(z; 1),\]
and it suffices to prove that the Fourier coefficients of $E^{*, Hil}_{m, k-m}(z; \ell)$ are rational for even $m$ and odd $\ell$ with $1\leq \ell<m\leq k/2$.
By Theorem \ref{RC}, $E^{*, Hil}_{m, k-m}(z; \ell)=C[E_{m +1-\ell}, E_{k+1-m-\ell}]^{Hil}_{\ell-1},$ where $C$ is a rational multiple of $\pi^{2-2\ell}$ by Theorem 9.8 on page 515 of \cite{N}. It follows that the Fourier coefficients of $E^{*, Hil}_{m, k-m}(z; \ell)$ belong to $\mathbb{Q}$.

Next, for primitive $f\in\mathcal H_k$, by Proposition 4.15 of \cite{Sh} and Theorem \ref{main1}, we have $\langle f, E^{*, Hil}_{k-1, 2}(z; k-1)\rangle=\alpha_f \langle f, \, f\rangle=\Lambda(f, k-1)\Lambda(f,k-2),$ for certain
$\alpha_f\in \mathbb{Q}(f)$. Again by Proposition 4.15, since $E^{*, Hil}_{k-1, 2}(z; k-1)$ has rational Fourier coefficients, $\alpha_{f}^\sigma=\alpha_{f^\sigma}$ for each $\sigma\in\text{Gal}(\bar{\mathbb{Q}}/\mathbb{Q})$. Also note  $\alpha_f \neq 0 $ because of the convergence of the Euler product of $\Lambda(f,s)$ for $\re(s)\geq k/2+1$ (see Kim-Sarnak's bound in \cite{KS}). Define 
\[\omega_+(f)=\frac{\alpha_f\langle f,f\rangle}{\Lambda(f,k-1)},\quad \omega_-(f)=\frac{\langle f,f\rangle}{\Lambda(f,k-2)}.\] Then for even $m$, odd $\ell$ with $1\leq m,\ell\leq k-1$, 
\[\frac{\Lambda(f,m)}{\omega_+(f)}=\frac{\langle f, E^{*, Hil}_{m, k-m}(z; k-1)\rangle}{\alpha_f\langle f,f\rangle}\in \mathbb{Q}(f)\]
again by Proposition 4.15 of \cite{Sh} and similarly $\frac{\Lambda(f,\ell)}{\omega_-(f)}\in\mathbb{Q}(f)$. It is clear that $\omega_+(f)\omega_-(f)=\langle f,f\rangle$. Finally, the assertion (4.16) of \cite{Sh} and that $\alpha_{f}^\sigma=\alpha_{f^\sigma}$ for each $\sigma\in\text{Gal}(\bar{\mathbb{Q}}/\mathbb{Q})$ implies that
\[\left(\frac{\Lambda(f,m)}{\omega_+(f)}\right)^\sigma=\frac{\Lambda(f^\sigma,m)}{\omega_+(f^\sigma)},\quad \left(\frac{\Lambda(f,\ell)}{\omega_-(f)}\right)^\sigma=\frac{\Lambda(f^\sigma,\ell)}{\omega_-(f^\sigma)},\]
finishing the proof.
$\Box$

\bibliographystyle{amsplain}

\begin{thebibliography}{10}

\bibitem{CO} Y. Choie, H. Kim and O. Richter, Differential operators on Hilbert modular forms, Journal of Number theory 122 (2007) 25-26.

\bibitem{CPZ} Y. Choie, Y. Park and D. Zagier, Periods of modular forms on $\Gamma_0(N)$ and products  of Jacobi theta functions, Journal of the European Mathematical Society 21, no. 5 (2019) 1379-1410.

\bibitem{C} H. Cohen, Sur certaines sommes de s\'eries li\'ees aux p\'erioded de formes modulaires, in S\'eminaire de th\'eorie de nombres, Grenoble, 1981.


\bibitem{CD1} N. Diamantis and C. O'Sullivan, Kernels of $L$-functions of cusp forms, Math. Ann. 346:4 (2010), 897-929.

\bibitem{CD2} N. Diamantis and C. O'Sullivan, Kernels for products of $L$-functions, Algebra and Number Theory 7:8 (2013) 1883-1917.

\bibitem{F} E. Freitag, Hilbert modular forms, Springer-Verlag, 2010.

\bibitem{G} P. B. Garrett, Holomorphic Hilbert modular forms, Brooks/Cole Publishing Company, 1990.

\bibitem{KS} H.  Kim and P. Sarnak, Refined estimates towards the Ramanujan and Selberg conjectures,
J. Amer. Math. Soc 16 (2003), no.1, 175-181.


\bibitem{KR} M. Knopp and S. Robins, Easy proofs of Riemann’s functional equation for $\zeta(s)$ and of Lipschitz summation, Proceedings of the American Mathematical Society 129 (7) (2001), 1915-1922.


\bibitem{KZ} W. Kohnen and D. Zagier, Modular forms with rational periods, Modular forms (Durham, 1983), 197-249, Ellis Horwood Ser. Math. Appl.: Statist. Oper. Res., Horwood, Chichester, 1984.


\bibitem{M} Y. Manin, Periods of parabolic forms and $p$-adic Hecke series, Mat. Sb. (N.S.), 21 (1973), 371-393.

\bibitem{N} J. Neukirch, Algebraic number theory, Springer-Verlag Berlin Heidelberg, 1999.

\bibitem{R} R. A. Rankin, The scalar product of modular forms, Proc. London Math. Soc. (3) 2 (1952), 198-217.


\bibitem{Sh} G. Shimura, The special values of the zeta functions associated with Hilbert modular forms, Duke Math. Journal 45 (1978), no.3, 637-679.

\bibitem{St} E. M. Stein and G. Weiss, Introduction to Fourier analysis on Euclidean spaces, Princeton University Press, 1971.


 \bibitem{Z1977} D. Zagier, Modular forms whose Fourier coefficients involve zeta-functions of quadratic fields, 105-169, Modular functions of one variable, VI (Bonn, 1976), edited by J-P. Serre and D. Zagier, Lecture note in Math 627, Springer, Berlin, 1977.

 \bibitem{Z} D. Zagier,  Modular forms and differential operators.  K. G. Ramanujan memorial issue. Proc. Indian Acad. Sci. Math. Sci.  104  (1994),  no. 1, 57-75.


\end{thebibliography}

\end{document}